\documentclass[11pt]{amsart}
\usepackage{amsmath,amsfonts,amsthm,amssymb,amsthm,graphicx,mathtools,amscd,enumerate,paralist,bm,hyperref}

\usepackage[mathscr]{eucal}
\usepackage{graphics,color,epsfig}
\usepackage[nocompress]{cite}

\DeclareSymbolFont{iwonaletters}{OML}{iwona}{m}{it}
\theoremstyle{plain}
\newtheorem{theorem}{Theorem}[section]
\newtheorem{lemma}[theorem]{Lemma}

\newtheorem{assumption}[theorem]{Assumption}

\newtheorem{remark}[theorem]{Remark}

\newcommand{\la}{\lambda}
\newcommand{\eps}{\varepsilon}

\newcommand{\al}{\alpha}

\newcommand{\gam}{\gamma}
\newcommand{\kap}{\kappa}
\newcommand{\sig}{\sigma}
\newcommand{\del}{\delta}

\newcommand{\Om}{\mathnormal{\Omega}}

\newcommand{\N}{{\mathbb N}}

\newcommand{\R}{{\mathbb R}}

\newcommand{\Z}{{\mathbb Z}}

\newcommand{\E}{{\mathbb E}}

\newcommand{\PP}{{\mathbb P}}

\newcommand{\calC}{{\mathcal C}}
\newcommand{\calD}{{\mathcal D}}

\newcommand{\calF}{{\mathcal F}}

\newcommand{\calL}{{\mathcal L}}

\newcommand{\calP}{{\mathcal P}}

\newcommand{\gW}{\grave{W}}
\newcommand{\gPi}{\grave{\Pi}}

\renewcommand{\proof}{\noindent{\bf Proof.\ }}

\newcommand{\lan}{\langle}
\newcommand{\ran}{\rangle}

\newcommand{\oo}{\overline}

\newcommand{\w}{\wedge}

\newcommand{\To}{\Rightarrow}

\newcommand{\iy}{\infty}

\newcommand{\cadlag}{c\`adl\`ag }

\newcommand{\noi}{\noindent}

\newcommand{\rs}{{\rm rs}}
\newcommand{\cs}{{\rm cs}}

\begin{document}

\title[]{
Mean field limits of large Jackson networks
\\
in heavy traffic
}

\author{Rami Atar}
\address{Viterbi Faculty of Electrical \& Computer Engineering
\\
Technion
} 
\email{rami@technion.ac.il}

%\subjclass[2010]{---}
%\keywords{---}

\date{\today}

\begin{abstract}
We consider an open Jackson network with $n$ exchangeable single-server stations and weak all-to-all interaction through routing: upon service completion at station $i$, a job is routed to station $j$ with probability $p/n$, where $p\in(0,1)$, or leaves the system with probability $q=1-p$. We study a joint asymptotic regime in which the number of stations tends to infinity while the system approaches heavy traffic. Under the critical-load condition and diffusive scaling of time and queue length, we prove propagation of chaos for the queue-length and cumulative-idleness processes. The limiting McKean--Vlasov dynamics are described by the nonlinear reflected Brownian motion
\[
\oo X(t)=\oo X_0+\oo W(t)+\hat\gamma t+\oo L(t)-p\,\E \oo L(t),
\]
where $\oo W$ is a Brownian motion with variance parameter $2$ and $\oo L$ is the reflection term at the origin. The proof proceeds by coupling the rescaled queueing network to a Brownian particle system interacting through boundary local times. A key step is a decoupling construction that replaces the correlated routing noise in the prelimit by asymptotically equivalent independent driving noises.
\end{abstract}

\subjclass[2020]{60K25, 60K35, 60J60, 60H10}

\keywords{Jackson networks; heavy traffic; McKean--Vlasov limits; propagation of chaos; reflected Brownian motion; mean-field routing}

\maketitle

\section{Introduction}

This paper studies a Jackson network with exchangeable nodes under a joint McKean--Vlasov and heavy-traffic scaling. The McKean--Vlasov, or many-station, limit corresponds to letting the number of nodes tend to infinity, whereas the diffusion limit corresponds to imposing a critical-load condition and applying diffusive scaling to the queue lengths. We show that, under this joint scaling, the diffusively scaled queue-length and idleness processes satisfy a propagation-of-chaos property and converge to a nonlinear reflected Brownian motion.

\subsection{Network structure and scaling}

The model consists of $n$ single-server queueing stations. Server $i$ is dedicated to queue $i$ and serves the job at the head of the line whenever the queue is nonempty. The queue length at station $i$ is the number of jobs in queue $i$, including the job in service, if any. The queue-length process is the $n$-dimensional process whose $i$th coordinate is the queue length at station $i$. Exogenous arrivals to each station are Poisson, and service times are exponential. Upon completion of service at station $i$, a job is routed to station $j$ with probability $p/n$, or leaves the system with probability $q$, where $p>0$, $q>0$, and $p+q=1$. Self-routing is allowed.

On a probability space $(\Om,\calF,\PP)$, we are given, for each $n$, an initial condition
$X^n_0=(X^n_{0i})_{i\in[n]}$, taking values in $\Z_+^n$, and a collection of Poisson processes
$(\Pi^n_{ij})_{(i,j)\in J^n}$, where
\[
J^n=(\{0\}\times[n])\cup[n]^2\cup([n]\times\{0\}).
\]
For $i\in[n]$, the process $\Pi^n_{0i}$ has intensity $\gamma_n$ and represents the stream of exogenous arrivals to station $i$. We assume that $\gamma_n$ is asymptotic to $\gamma+\hat\gamma n^{-1/2}$, where $\gamma>0$ and $\hat\gamma\in\R$; that is,
\[
\lim_{n\to\infty}\hat\gamma_n=\hat\gamma,
\qquad
\hat\gamma_n=n^{1/2}(\gamma_n-\gamma).
\]
For $i\in[n]$, the process $\Pi^n_{ij}$ has rate $p/n$ for $j\in[n]$, including the case $j=i$, and has rate $q$ for $j=0$. These processes represent potential service completions at station $i$, followed either by routing to station $j\in[n]$ or by departure from the system when $j=0$. Such a potential completion is counted as an actual completion only if queue $i$ is nonempty at the event time. The Poisson processes are assumed to be mutually independent and independent of $X^n_0$. For each $i\in[n]$, the total rate of the potential service-completion processes $\Pi^n_{ij}$, $j=0,1,\ldots,n$, is $p+q=1$, corresponding to a unit service rate at each station.

Let $X^n=(X^n_i)_{i\in[n]}$ be the queue-length process, modeled as the unique solution of
\begin{equation}\label{q1}
\begin{split}
X^n_i(t)
&=X^n_{0i}+\Pi^n_{0i}(t)
-\sum_{j=0}^n\int_0^t B^n_i(s-)\,d\Pi^n_{ij}(s)
+\sum_{j=1}^n \int_0^t B^n_j(s-)\,d\Pi^n_{ji}(s),
\qquad i\in[n],
\\
B^n_i(t)&=1_{\{X^n_i(t)>0\}}, \qquad i\in[n].
\end{split}
\end{equation}
The first two terms on the right are the initial condition and the exogenous arrivals to station $i$. The third term counts actual service completions from station $i$, including both routed jobs and departures from the network, while the fourth term counts jobs routed into station $i$ from other stations. The busyness indicators $B^n_i$ ensure that potential service completions are counted only when the corresponding queue is nonempty.
Define also the idleness indicator $I^n_i=1-B^n_i$ and the cumulative idleness process
\[
L^n_i(t)=\int_0^t I^n_i(s)\,ds.
\]
%Thus \eqref{q1} may be viewed as the balance equation
%\[
%X^n_i=X^n_{0i}+\Pi^n_{0i}-D^n_i+A^n_i,
%\qquad i\in[n],
%\]
%where $D^n_i$ and $A^n_i$ denote, respectively, the cumulative departures from station $i$ and the cumulative internal arrivals to station $i$.

To obtain a critically loaded system, we impose the heavy-traffic condition
\begin{equation}\label{e-ht}
\gamma=q.
\end{equation}
We study the diffusively scaled queue-length and cumulative-idleness processes
\begin{equation}\label{x3}
\hat X^n_i(t)=n^{-1/2}X^n_i(nt),
\qquad
\hat L^n_i(t)=n^{-1/2}L^n_i(nt),
\qquad i\in[n].
\end{equation}
Finally, define the empirical measure of the scaled queue-length/idleness pairs and the empirical mean of the scaled cumulative idleness by
\begin{equation}\label{45}
\mu^n_t=\frac{1}{n}\sum_{i\in[n]}\delta_{(\hat X^n_i(t),\hat L^n_i(t))},
\qquad
\lambda^n(t)=\frac{1}{n}\sum_{i\in[n]}\hat L^n_i(t).
\end{equation}

\subsection{The limiting McKean--Vlasov process}

A one-dimensional Brownian motion with drift $b$ and variance $\sigma^2$ will be referred to as a $(b,\sigma^2)$-BM. The McKean--Vlasov process relevant to our setting, also referred to as the {\it nonlinear reflected Brownian motion} (RBM), is given by
\begin{equation}\label{MVSDE}
\begin{split}
\oo X(t) &= \oo X_0+\oo W(t)+\hat\gamma t+\oo L(t)-p\,\lambda(t),
\\
\lambda(t) &= \E \oo L(t).
\end{split}
\end{equation}
Here, $\oo W$ is a $(0,2)$-BM independent of the initial condition $\oo X_0$, and $\oo L$ is the reflection term at $0$. Thus $\oo X$ takes values in $\R_+:=[0,\iy)$, $\oo L$ is nondecreasing, $\oo L(0)=0$, and $\oo L$ increases only when $\oo X=0$. By \cite[Theorem 2.4]{atar26}, \eqref{MVSDE} admits a pathwise unique strong solution, and the processes $\oo X$ and $\oo L$ have continuous paths.

The form of \eqref{MVSDE} is natural in view of classical heavy-traffic theory for Jackson networks. Indeed, a classical result of Reiman \cite{reiman1984open} shows that, when $n$ is fixed, the diffusion-scaled queue-length process converges in heavy traffic to a semimartingale RBM in $\R_+^n$, with reflection matrix determined by the routing probabilities. In the present model, routing probabilities are of order $1/n$, and therefore the corresponding reflection interaction is weak at the level of each pair of stations but is spread over all stations. The nonlinear RBM \eqref{MVSDE} is the mean-field analogue of this fixed-dimensional reflected Brownian approximation.

The proof uses recent results about Brownian particle systems on the half-line. Let $(\tilde X^n,\tilde L^n)$ be the solution of
\begin{equation}\label{x44}
\tilde X^n_i(t)
=
\tilde X^n_{0i}
+\tilde W_i(t)
+\hat\gamma t
+\tilde L^n_i(t)
-\frac{p}{n}\sum_{j\in[n]}\tilde L^n_j(t),
\qquad i\in[n],
\end{equation}
where $(\tilde W_i)_{i\in[n]}$ are i.i.d.\ $(0,2)$-BMs and the processes $\tilde L^n_i$ are the corresponding reflection terms. The results of both \cite{atar26} and \cite{baker2025particle} show that this Brownian system undergoes propagation of chaos, with \eqref{MVSDE} as its McKean--Vlasov limit
(a minor normalization difference is addressed in Section \ref{sec23}). Our main task is to show that, under the joint heavy-traffic and many-station scaling, the queueing network can be coupled to this Brownian particle system with an error that vanishes as $n\to\infty$.

The connection with \eqref{x44}, however, is not automatic from the fixed-dimensional heavy-traffic theorem. The semimartingale RBM describing the fixed-dimensional limit is driven by a Brownian motion with dependent components, whereas \eqref{x44} is driven by i.i.d.\ BMs. One of the main steps of the proof is to overcome this dependence by an asymptotically negligible perturbation of the prelimit driving noise. The mechanism behind this perturbation is described after Theorem \ref{th1} and carried out in Lemma \ref{lem21}.

\subsection{Main result}

Some notation needed to present the main result is as follows.
For a Polish space $E$, let $C(\R_+,E)$ be the space of continuous paths $\R_+\to E$, equipped with the local uniform topology, and let $D(\R_+,E)$ be the space of \cadlag paths $\R_+\to E$, equipped with the $J_1$ topology. In particular, for $n\in\N$ let
\[
C^{(n)}=C(\R_+,\R^n),
\qquad
D^{(n)}=D(\R_+,\R^n).
\]
Let $\calP(E)$ denote the set of probability measures on $E$, equipped with the topology of weak convergence. Write $\To$ for convergence in distribution.

The initial conditions are assumed to satisfy the following.

\begin{assumption}\label{assn1}
For each $n$, the initial conditions $(\hat X^n_{0i})_{i\in[n]}$ are exchangeable.
Moreover, there exists $\mu_0^{(1)}\in\calP(\R_+)$ such that
\[
\frac{1}{n}\sum_{i\in[n]}\delta_{\hat X^n_{0i}}\to\mu_0^{(1)}
\]
in $\calP(\R_+)$ in probability.
\end{assumption}

Our main result is the following.

\begin{theorem}\label{th1}
Let Assumption \ref{assn1} hold. Let $k\in\N$, and let $(\oo X,\oo L,\lambda)$ be the solution of \eqref{MVSDE} with $\oo X_0\sim\mu_0^{(1)}$. Then
\[
((\hat X^n_i)_{i\in[k]},(\hat L^n_i)_{i\in[k]})
\To
((\oo X_i)_{i\in[k]},(\oo L_i)_{i\in[k]})
\]
in $D^{(k)}\times C^{(k)}$, where $(\oo X_i,\oo L_i)$ are $k$ independent copies of $(\oo X,\oo L)$. Moreover,
\[
(\mu^n,\lambda^n)\to(\mu,\lambda)
\]
in $D(\R_+,\calP(\R_+^2))\times C^{(1)}$, in probability, where
\[
\mu_t=\PP\circ(\oo X(t),\oo L(t))^{-1},
\qquad t\ge0.
\]
\end{theorem}

The theorem shows that the joint heavy-traffic and many-station limit is described by independent copies of the nonlinear RBM \eqref{MVSDE}, and the empirical distribution of the scaled queue-length/idleness pairs converges to the law of the McKean--Vlasov process. As mentioned above, the proof proceeds by coupling the queueing network to the Brownian particle system \eqref{x44}. Two estimates are needed for this purpose. The first shows that an idleness-dependent martingale term in the queueing dynamics is negligible. The second addresses the dependence in the prelimit driving noise. This dependence comes from the row and column sums of a common Poisson routing array. We construct an independent replacement for the column-sum process, with the same law as the original, whose difference from the original column-sum process is negligible after diffusion scaling. This turns the correlated prelimit driving processes into independent martingales whose limits are independent BMs. These two ingredients reduce the proof to the propagation-of-chaos result for \eqref{x44} established in \cite{atar26}.

\begin{remark}
Unlike standard McKean--Vlasov diffusions, where the interaction acts through the state
variable, the interaction in \eqref{MVSDE} is carried by the reflection process. This
originates from the queueing network: $L_i^n(t)$ is the cumulative idle time of server $i$,
and the routing mechanism couples stations through the empirical average
$\lambda^n(t)$ (this is apparent in equation \eqref{x30}). In the heavy-traffic limit, idleness is encoded by the reflection term,
leading to the interaction term $pE[\oo L(t)]$ in \eqref{MVSDE}.
\end{remark}

\begin{remark}
One common use of diffusion approximations for queueing networks is to study
stationary behavior by first passing to a scaling limit and then sending time to
infinity in the limiting model.
This work is not aimed at deriving stationary approximations.
Indeed, for the model studied here, the stationary distribution of the prelimit
is explicit. If $\hat\gamma_n<0$, or equivalently $\gamma_n<q$, then the traffic
intensity at each station is
\[
\rho_n=\frac{\gamma_n}{q},
\]
and the product-form stationary distribution is given by
\[
\pi^n(x)=\prod_{i\in[n]} (1-\rho_n)\rho_n^{x_i},
\qquad x=(x_1,\ldots,x_n)\in\mathbb Z_+^n.
\]
In particular, when $\hat\gamma_n\to\hat\gamma<0$, the stationary law of
$n^{-1/2}X_i^n$ converges to the exponential distribution with rate
$|\hat\gamma|/q$. Thus the interest here is not in deriving new stationary
distributions, but rather in describing the time-dependent, or transient,
behavior of the network under the scaling.
\end{remark}

\subsection{Related work}

The fixed-dimensional heavy-traffic theory for generalized Jackson networks goes back to the aforementioned work \cite{reiman1984open}, which proves convergence of diffusion-scaled queue-length processes to semimartingale RBM in the orthant. We refer to \cite{chen-yao} for background on RBM in orthants, the associated Skorohod problem, queueing networks, and their diffusion approximations.

McKean--Vlasov limits and propagation of chaos originate with the pioneering work of McKean \cite{mckean1966class} and were developed systematically in \cite{gar88} and \cite{sznitman1991topics}. McKean--Vlasov limits for reflected diffusions were first studied in \cite{sznitman1984nonlinear}, and many extensions have since been considered. In much of this literature, the interaction acts through the empirical law in the drift or diffusion coefficients, whereas in the present setting the interaction arises through the boundary, via the cumulative idleness or local-time terms.

Particle systems with interaction through local times are closer in spirit to the limit appearing here. Such models appear in \cite[Chapter II]{sznitman1991topics}, where particles interact through their collision local times, and in recent works on systems with constraints on the empirical law; see \cite{coghi2022mckean} and references therein.

The propagation-of-chaos input used in the present paper comes from \cite{atar26}, which studies Brownian particle systems on the half-line with interaction through boundary local times and proves the corresponding limit theorem for the reference system \eqref{x44}, and also treats a more general model with random reflection directions. The system \eqref{x44} had previously been studied in \cite{baker2025particle}, where its McKean--Vlasov limit was established; the emphasis there is on regimes in which solutions may exist only locally in time, including cases such as $p \ge 1$.

There is also a substantial body of work on mean-field limits for queueing networks. In this direction, one lets the number of servers, stations, or nodes tend to infinity at the queueing-network level, without applying diffusion scaling to the queue lengths. Examples include \cite{baccelli1992mean,rybko2005poisson2,rybko2016stationary}. These works are distinct from the present heavy-traffic regime and therefore from the Brownian approximation obtained here.

\subsection{Further notation}

For $x\in\R$, let $x^{\pm}=\max(\pm x,0)$. Let $\iota$ denote the identity map on $\R_+$, $\iota(t)=t$. For $n\in\N$, equip $\R^n$ with the Euclidean norm $|\cdot|$.

For a Polish space $(E,d)$ and $G\subset E$, let
\[
C_G(\R_+,E)=\{x\in C(\R_+,E):x(0)\in G\},
\]
and define $D_G(\R_+,E)$ similarly. Let $C_+(\R_+,\R_+)$ denote the set of functions in $C(\R_+,\R_+)$ that are nondecreasing and start at $0$, and define $D_+(\R_+,\R_+)$ similarly.

Let $\calC^{(n)}$ and $\calD^{(n)}$ denote the Borel $\sig$-algebras on $C^{(n)}$ and $D^{(n)}$, respectively.

For $f:\R_+\to\R^n$ and $0\le\delta\le t$, let
\[
\|f\|_t=\sup\{|f(s)|:s\in[0,t]\}
\]
and
\[
w_t(f,\delta)=\sup\{|f(s)-f(u)|:s,u\in[0,t],\, |s-u|\le\delta\}.
\]
Finally, for any $K=(K_i)_{i\in[n]}$, write
\[
\lan K\ran:=\frac{1}{n}\sum_{i\in[n]}K_i.
\]

\subsection{Organization of the paper}

Section \ref{sec2} gives several tools used in the proof. Section \ref{sec21} derives an equation for the diffusively scaled queueing network, Section \ref{sec22} recalls the Skorohod map on the half-line, and Section \ref{sec23} states the propagation-of-chaos result for the Brownian particle system from \cite{atar26}. Section \ref{sec3} proves the main theorem. The proof is based on two ingredients: an estimate showing that an idleness-dependent martingale term is negligible, and a decoupling construction that replaces the correlated routing noise by independent BMs. These ingredients are then combined to prove Theorem \ref{th1}.

\section{Toolbox}\label{sec2}

\subsection{An equation for the rescaled system}\label{sec21}

Denote accelerated processes by
\[
\check\Pi^n_{ij}(t)=\Pi^n_{ij}(nt), \qquad
\check B_i(t)=B_i(nt), \qquad
\check I_i(t)=I_i(nt).
\]
Note that $\check\Pi^n_{ij}$ are Poisson processes with rates $\gamma_nn$, $p$, $qn$.
Let compensated versions be given by
\begin{alignat*}{2}
\hat\Pi_{0i}^n(t)&=n^{-1/2}(\check\Pi_{0i}(t)-\gamma_nnt)
\qquad &&i\in[n]
\\
\hat\Pi_{ij}^n(t)&=n^{-1/2}(\check\Pi_{ij}(t)-pt)
 &&i\in[n],\, j\in[n]
\\
\hat\Pi_{i0}^n(t)&=n^{-1/2}(\check\Pi_{i0}(t)-qnt)
&& i\in[n].
\end{alignat*}
Our goal here is to write the dynamics of $\hat X_i$ in terms of $\hat\Pi_{ij}$, $\check B_i$, and $\hat L_i$ alone. This is achieved in equations \eqref{x30}-\eqref{x31} below.
To this end, note that by \eqref{q1},
\begin{align}\label{k1}
\hat X_i(t)=n^{-1/2}X_i(nt) &=\hat X_{0i}+n^{-1/2}\check\Pi_{0i}(t)-n^{-1/2}\sum_{j=0}^n\int_0^{t} \check B_i(s-)d\check\Pi_{ij}(s)
\\ \notag
&\qquad +n^{-1/2}\sum_{j=1}^n \int_0^{t}\check B_j(s-) d\check\Pi_{ji}(s).
\end{align}
If we let $G_{ij}(t)=n^{-1/2}\int_0^{t}\check B_i(s-)d\check\Pi_{ij}(s)$ then for $i\in[n]$, $j\in[n]$,
\begin{align*}
G_{ij}(t)
&=\hat\Pi_{ij}(t)+pn^{-1/2}t-n^{-1/2}\int_0^{t}\check I_i(s-)d\check\Pi_{ij}(s)
\\
&=\hat\Pi_{ij}(t)+pn^{-1/2}t-\int_0^t\check I_i(s-)d\hat\Pi_{ij}(s)-pn^{-1}\hat L_i(t),
\end{align*}
whereas
\begin{align*}
G_{i0}(t)
&=\hat\Pi_{i0}(t)+qn^{1/2}t-n^{-1/2}\int_0^{t}\check I_i(s-)d\check\Pi_{i0}(s)
\\
&=\hat\Pi_{i0}(t)+qn^{1/2}t-\int_0^t\check I_i(s-)d\hat\Pi_{i0}(s)-q\hat L_i(t).
\end{align*}
Now, the sum of the first two terms on the right side of \eqref{k1} equals
\[
\hat X_{0i}+\hat\Pi_{0i}(t)+\gamma_nn^{1/2}t
=\hat X_{0i}+\hat\Pi_{0i}(t)+\hat\gamma_nt+n^{1/2}\gamma t.
\]
For the remaining terms on the right side of \eqref{k1},
\[
-\sum_{j\in[n]}G_{ij}=-\sum_{j\in[n]}\hat\Pi_{ij}(t)-pn^{1/2}t+\sum_{j\in[n]}\int_0^t\check I_i(s-)d\hat\Pi_{ij}(s)+p\hat L_i(t),
\]
while
\[
\sum_{j\in[n]}G_{ji}=
\sum_{j\in[n]}\hat\Pi_{ji}(t)
+pn^{1/2}t-\sum_{j\in[n]}\int_0^t\check I_j(s-)d\hat\Pi_{ji}(s)-pn^{-1}\sum_{j\in[n]}\hat L_j(t).
\]
Recalling \eqref{e-ht}, by which $\gamma=1-p=q$, this gives
\begin{align}\label{x30}
\hat X_i&=\hat X_{0i}+\hat W_i+\hat\gamma_n\iota+\hat L_i-p\lan\hat L\ran+\hat M_i,
\end{align}
where
\begin{equation}\label{x33}
\hat W_i=\hat\Pi_{0i}-\sum_{j\ge0}\hat\Pi_{ij}+\sum_{j\ge1}\hat\Pi_{ji},
\end{equation}
\begin{equation}\label{x46}
\hat M_i=
\sum_{j\ge0}\hat M_{ij}-\sum_{j\ge1}\hat M_{ji},
\qquad
\hat M_{ij}(t)=\int_0^t\check I_i(s-)d\hat\Pi_{ij}(s).
\end{equation}
Also, by the definition of $\hat L_i$, one has the complementarity condition
\begin{equation}\label{x31}
\int\hat X_id\hat L_i=0.
\end{equation}

\subsection{The Skorohod map}\label{sec22}

The {\it Skorohod problem on the half line} is stated as follows:

{\it Given $w\in D_{\R_+}(\R_+,\R)$, find $x\in D(\R_+,\R_+)$ and $\ell\in D_+(\R_+,\R_+)$ such that $x=w+\ell$ and $\int_{[0,\iy)}xd\ell=0$.
}

An elementary lemma provides its unique solution.
\begin{lemma}\label{lem02}
Given $w\in D_{\R_+}(\R_+,\R)$ there exists a unique solution $(x,\ell)\in D(\R_+,\R_+)\times D_+(\R_+,\R_+)$ to the Skorohod problem on the half line. It is given by
\begin{equation}
\label{x41}
x(t)=w(t)+\sup_{s\le t}w^-(s),
\qquad
\ell(t)=\sup_{s\le t}w^-(s),
\qquad t\ge0.
\end{equation}
\end{lemma}
\proof See \cite[Theorem 6.1]{chen-yao}.
\qed

Denote the solution map by $\Gamma:D_{\R_+}(\R_+,\R)\to D(\R_+,\R_+)\times D_+(\R_+,\R_+)$. Here, $\Gamma(w)=(\Gamma_1(w),\Gamma_2(w))=(x,\ell)$ is given by \eqref{x41}.

Note that the map is Lipschitz in the sense that, for
$w_1,w_2\in D_{\R_+}(\R_+,\R)$ and $t\in\R_+$,
\begin{equation}\label{x52}
\|\Gamma_1(w_1)-\Gamma_1(w_2)\|_t\le 2\|w_1-w_2\|_t,
\qquad
\|\Gamma_2(w_1)-\Gamma_2(w_2)\|_t\le \|w_1-w_2\|_t.
\end{equation}
That $\Gamma_2$ is $1$-Lipschitz will play an important role in our proofs.

Note that with the notation just introduced we can write \eqref{x30}--\eqref{x31} as follows. Let
\begin{equation}
\label{x39}
\hat Y_i=\hat W_i+\hat\gam_n\iota+\hat M_i.
\end{equation}
Then
\begin{equation}
\label{x42}
(\hat X_i,\hat L_i)=\Gamma(\hat X_{0i}+\hat Y_i-p\lan\hat L\ran),\qquad i\in[n].
\end{equation}

We now show that \eqref{x42} has a unique solution $(\hat X^n_i,\hat L^n_i)$ for given data $\hat X_{i0}+\hat Y_i$.

\begin{lemma}\label{lem22}
Fix $n\in\N$ and $\beta\in(-1,1)$. Let $w\in D_{\R_+^n}(\R_+,\R^n)$. Then the system of equations
\[
(x_i,\ell_i)=\Gamma(w_i+\beta\lan\ell\ran),
\qquad i\in[n]
\]
has a unique solution $(x,\ell)\in D(\R_+,\R_+)^n\times D_+(\R_+,\R_+)^n$.
\end{lemma}

\proof
A proof of this result was given in \cite{har-rei} for a more general class of Skorohod problems in the orthant, assuming continuous data $w$. Their argument is based on a contraction mapping. The same argument applies without change to \cadlag paths $w$, and is sufficient for our purposes here. For completeness, we include the proof in the present special case, where the argument is particularly simple.

It suffices to consider the problem on a finite time interval $[0,T]$. Let $D^{(n)}_T=D([0,T],\R)^n$.
Fix $w$ as in the statement of the lemma and let $\Phi:D^{(n)}_T\to D^{(n)}_T$ be the map
\[
\Phi_i(\ell)=\Gamma_2(w_i+\beta\lan\ell\ran)=\sup_{s\le\cdot}[(w_i(s)+\beta\lan\ell\ran(s))^-],
\qquad
\Phi(\ell)=(\Phi_i(\ell))_{i\in[n]}.
\]
For $\ell,\bar\ell\in D^{(n)}_T$, we have
\[
\max_i\|\Phi_i(\ell)-\Phi_i(\bar\ell)\|_T
\le |\beta|\,\|\lan\ell\ran-\lan\bar\ell\ran\|_T
\le|\beta|\,\lan \|\ell-\bar\ell\|_T\ran\le|\beta|\max_i\|\ell_i-\bar\ell_i\|_T.
\]
Thus $\Phi$ is a contraction on $D^{(n)}_T$ equipped with the uniform metric, which is a complete metric space, and therefore $\Phi$ has a unique fixed point $\ell$. The components of $\ell$ are nondecreasing owing to the running supremum expression in \eqref{x41}. Finally, again by \eqref{x41}, it is easy to see that $x=w+\ell$ takes values in $\R_+^n$.
\qed

\subsection{The Brownian system}\label{sec23}

Here we present a result from \cite{atar26} for a Brownian particle system.

Let $(\tilde X^n_{0i})_{i\in[n]}$ be exchangeable $\R_+$-valued RVs for every $n$, serving as initial conditions, and satisfying $n^{-1}\sum_{i\in[n]}\del_{\tilde X^n_{0i}}\to\mu_0^{(1)}$ in $\calP(\R_+)$ in probability. Let $(\tilde W_i)_{i\in[n]}$ be mutually independent $(0,2)$-BMs, independent of the initial conditions. Applying Lemma \ref{lem22}, let $(\tilde X^n_i,\tilde L^n_i)$ be the unique solution to
\[
(\tilde X^n_i,\tilde L^n_i)=\Gamma(\tilde X^n_{0i}+\tilde W^n_i+\hat\gam\iota-p\lan\tilde L^n\ran).
\]
Clearly, the sample paths of $\tilde X^n$ and $\tilde L^n$ lie in $C(\R_+,\R_+)^n$ and $C_+(\R_+,\R_+)^n$, respectively.
Define, similar to \eqref{45},
\[
\tilde\mu^n_t=\frac{1}{n}\sum_{i\in[n]}\del_{(\tilde X^n_i(t),\tilde L^n_i(t))},
\qquad
\tilde\la^n(t)=\frac{1}{n}\sum_{i\in[n]}\tilde L^n_i(t).
\]

\begin{theorem}[\cite{atar26}]
\label{th2}
Let $(\oo X,\oo L,\la)$ be a solution to \eqref{MVSDE} with $\oo X_0\sim\mu_0^{(1)}$. Let $k\in\N$. Then $(\tilde X^n_i,\tilde L^n_i)_{i\in[k]}\To(\oo X_i,\oo L_i)_{i\in[k]}$ in $C^{(2k)}$, where $(\oo X_i,\oo L_i)$ are $k$ independent copies of $(\oo X,\oo L)$. Moreover, $(\tilde\mu^n,\tilde\la^n)\to(\mu,\la)$ in $C(\R_+,\calP(\R_+^2))\times C^{(1)}$, in probability, where $\mu_t=\PP\circ (\oo X(t),\oo L(t))^{-1}$ for $t\ge0$.
\end{theorem}

\proof
In \cite{atar26}, a Brownian particle system is considered in which the boundary terms are averaged in a slightly different way, namely
\begin{equation}\label{01a}
Y_i^n(t)=\tilde X^n_{0,i}+\tilde W_i^n(t)+\hat\gamma t+K_i^n(t)-\frac{p}{n-1}\sum_{j\ne i}K_j^n(t),
\qquad i\in[n], \ t\ge0,
\end{equation}
with $K_i$ the reflection terms.
The results show precisely the statement of the theorem for $(Y^n_i,K^n_i)$. Hence it suffices to show that the difference is negligible. By exchangeability, it suffices to prove that $\|\tilde X^n_1-Y^n_1\|_t\to0$ and $\|\tilde L^n_1-K^n_1\|_t\to0$ in probability. To this end, using \eqref{x52},
\begin{align*}
\|\tilde L_i-K_i\|_t &\le p\Big\|\frac{1}{n-1}\sum_{j\ne i}\tilde L_j-\frac{1}{n}\sum_{j}K_j\Big\|_t
\\
&= p\Big\|\frac{1}{n}\sum_j(\tilde L_j-K_j)+\Big(\frac{1}{n}-\frac{1}{n-1}\Big)\sum_j\tilde L_j-\frac{1}{n-1}\tilde L_i\Big\|_t.
\end{align*}
Hence, with $\eps_n=(n-1)^{-1}$,
$\|\tilde L_i-K_i\|_t\le p\lan \|\tilde L-K\|_t\ran+\eps_n\lan\tilde L\ran(t)+\eps_n\tilde L_i(t)$. Thus
\[
\E\|\tilde L_1-K_1\|_t\le 2\eps_n(1-p)^{-1}\E\tilde L_1(t)
\le 2\eps_n(1-p)^{-1}\{\E\|\tilde L_1-K_1\|_t+\E K_1(t)\}.
\]
It follows from \cite[Lemma 3.6]{atar26} that $\E K_1(t)$ remains bounded as $n\to\iy$. The above thus shows that $\E\|\tilde L_1-K_1\|_t\to0$.
The result $\E\|\tilde X_1-Y_1\|_t\to0$ follows.
\qed

\section{Proof of main result}\label{sec3}

We now state the two main lemmas on which the proof of Theorem \ref{th1} is based.
First is an estimate on the martingale term in \eqref{x30}.

\begin{lemma}\label{lem20}
$\lim_n\E[\|\hat M_1\|_t^2]=0$.
\end{lemma}

Next, let us rewrite $\hat W=\hat W^n$ of \eqref{x33} in a way that emphasizes that its structure involves the row-sum and column-sum of an array of compensated Poissons:
\begin{equation}\label{x34}
\hat W_i=\hat\Pi_{0i}-\hat\Pi_{i0}-\hat\Pi^\rs_i+\hat\Pi^\cs_i,
\qquad
\hat\Pi^\rs_i:=\sum_{j\in[n]}\hat\Pi_{ij},
\qquad
\hat\Pi^\cs_i:=\sum_{j\in[n]}\hat\Pi_{ji}.
\end{equation}
Clearly, the components $\hat W_i$ are dependent.
The following lemma will allow us construct a perturbation of $\hat W_i$ whose components are independent.

\begin{lemma}\label{lem21}
After augmenting the space with additional Poisson processes, if necessary, there exists, for each $n$, a tuple $(\Pi^*_i)_{i\in[n]}$ of i.i.d.\ rate-$pn$ Poisson processes that are independent of $(\hat\Pi_{0i},\hat\Pi_{i0},\hat\Pi^\rs_i)_{i\in[n]}$, such that, denoting $\gPi_i=n^{-1/2}(\Pi^*_i-pn\iota)$, the processes $(\gPi_i-\hat\Pi^\cs_i)_{i\in[n]}$ are exchangeable and $\E[\|\gPi_1-\hat\Pi^\cs_1\|_t]\to0$ as $n\to\iy$.
\end{lemma}

In Sections \ref{sec31} and \ref{sec32}, we present the proofs of Lemmas \ref{lem20} and \ref{lem21}, respectively. In Section \ref{sec33}, we give the proof of Theorem \ref{th1}.

\subsection{Proof of Lemma \ref{lem20} (Estimating the martingale term)}
\label{sec31}

We fix $t$ and let $c$ denote a positive constant that does not depend on $n$ (but may depend on $t$), whose value may change from one appearance to another.

The proof proceeds in two steps, establishing the following two estimates:
\begin{equation}\label{b1}
\E[\|\hat M_1\|^2_t] \le cn^{-1/2}\E\hat L_1(t)
\end{equation}
\begin{equation}\label{b2}
\E\hat L_1(t)\le c+c\E\|\hat M_1\|_t.
\end{equation}
Together, these estimates imply the result. Indeed, the inequality $x^2\le cn^{-1/2}(1+x)$ implies $x^2\le cn^{-1/2}$.

Toward proving \eqref{b1}, we have by \eqref{x46},
\begin{align*}
\hat M_1(t) &= \hat M_{10}(t)+\sum_{j\in[n],j\ne 1}(\hat M_{1j}(t)-\hat M_{j1}(t))
\\
&=\int_0^t\check I_1(s-)d\hat\Pi_{10}(s)
+\sum_{j\in[n],j\ne 1}\Big(\int_0^t\check I_1(s-)d\hat\Pi_{1j}(s)-\int_0^t\check I_j(s-)d\hat\Pi_{j1}(s)\Big).
\end{align*}
Because $\hat\Pi_{10}$, $\hat\Pi_{1j}$, $\hat\Pi_{j1}$, $j\ne1$ do not have common jump times, we have
\[
[\hat M_1](t)=n^{-1}\int_0^t\check I_1(s-)d\check\Pi_{10}(s)
+n^{-1}\sum_{j\in[n],j\ne 1}\Big(\int_0^t\check I_1(s-)d\check\Pi_{1j}(s)+\int_0^t\check I_j(s-)d\check\Pi_{j1}(s)\Big).
\]
Recalling that the compensators of $\check\Pi_{10}$, $\check\Pi_{1j}$, $\check\Pi_{j1}$ are $qnt$, $pt$ and $pt$, respectively, and using exchangeability,
\begin{align*}
\E\{[\hat M_1](t)\}&=\E\int_0^t\check I_1(s)q\,ds
+n^{-1}\sum_{j\in[n],j\ne 1}\E\int_0^t(\check I_1(s)+\check I_j(s))p\,ds
\\
&= (q+2pn^{-1}(n-1))\E\int_0^t\check I_1(s)ds
\\
&\le 2\,\E\int_0^t\check I_1(s)ds.
\end{align*}
Noting that $\int_0^t\check I_1(s)ds=n^{-1/2}\hat L_1(t)$ and using Doob's inequality gives \eqref{b1}.

Next, to show \eqref{b2}, note that
by \eqref{x42}, Lemma \ref{lem02}, and then \eqref{x39},
\begin{align*}
\hat L_1(t)&=\Gamma_2(\hat X_{01}+\hat Y_1-p\lan\hat L\ran)(t)
\\
&=\sup_{s\le t}[(\hat X_{01}+\hat Y_1(s)-p\lan\hat L\ran(s))^-]
\\
&\le\sup_{s\le t}[(\hat Y_1(s)-p\lan\hat L\ran(s))^-]
\\
&\le
\|\hat W_1\|_t+|\hat\gam_n|t+\|\hat M_1\|_t+p\lan\hat L\ran(t),
\end{align*}
where $\|\lan\hat L\ran\|_t=\lan\hat L\ran(t)$ because $\hat L_i$ are nondecreasing. Taking expectation, using exchangeability and the fact that $p<1$ gives
\begin{align}
\label{x45}
\E\hat L_1(t)\le c+c\E\|\hat W_1\|_t+c\E\|\hat M_1\|_t.
\end{align}
By \eqref{x33},
\[
\hat W_1=\hat\Pi_{01}-\hat\Pi_{10}+\sum_{j\in[n],j\ne 1}(\hat\Pi_{j1}-\hat\Pi_{1j}).
\]
In particular, $\hat W_1$ is a martingale and
\[
[\hat W_1](t)=n^{-1}\Big(\check\Pi_{01}(t)+\check\Pi_{10}(t)+\sum_{j\in[n],j\ne 1}(\check\Pi_{j1}(t)+\check\Pi_{1j}(t))\Big).
\]
The expression in parentheses is a Poisson RV of parameter $(n\gam_n+nq+2(n-1)p)t$, implying that $\sup_n\E\{[\hat W_1](t)\}<\iy$. As a result, $\sup_n\E\|\hat W_1\|_t<\iy$. Using this in \eqref{x45} gives \eqref{b2}.
\qed

\subsection{Proof of Lemma \ref{lem21} (Decoupling the routing noise)}
\label{sec32}

The proof is based on a construction stated in Lemma \ref{lem14} below. To present this lemma, let, for each $n\in\N$, $N_{ij}=N^{(n)}_{ij}$, $(i,j)\in[n]^2$, be an array of i.i.d.\ rate-$\beta$ Poisson processes, for some constant $\beta>0$. Let
\[
U_i=\sum_jN_{ij}, \quad i\in[n],\qquad
V_j=\sum_iN_{ij}, \quad j\in[n].
\]
Then $(U_i)_{i\in[n]}$ are independent rate-$\beta n$ Poisson processes, and so are $V_j$, $j\in[n]$. 
The families $(U_i)_{i\in[n]}$ and $(V_i)_{i\in[n]}$ are each composed of independent Poisson processes, but they are not independent of one another since they are built from the same Poisson array. The purpose of the lemma is to construct a new column-sum process $Z$, independent of the row sums $U$, which has the same law as $V$ and remains asymptotically close to $V$ at the scale needed for the diffusion approximation.

\begin{lemma}\label{lem14}
After augmenting the probability space to support additional Poisson processes, if necessary, there exist, for each $n\in\N$, independent rate-$\beta n$ Poisson processes $Z^{(n)}_j$, $j\in[n]$ such that $Z^{(n)}=\{Z^{(n)}_j\}$ is independent of $U^{(n)}=\{U^{(n)}_i\}$.
Moreover, the processes $(V^{(n)}_j,Z^{(n)}_j)$, $j\in[n]$, are exchangeable, and for any $t$ and any $\kap>0$,
\[
n^{-\kap}\E[\|Z^{(n)}_1-V^{(n)}_1\|_t]\to0.
\]
\end{lemma}

\proof
Let $S=\sum_{ij}N_{ij}$. It is a rate-$\beta n^2$ Poisson process. Let $\sig_k$, $k\ge1$ denote the jump times of $S$ in increasing order. Let $(R_k,C_k)$ be the row and column labels associated with the $k$-th arrival. Namely, $(R_k,C_k)$ is the unique $(i,j)$ for which $N_{ij}$ makes a jump at time $\sig_k$. Clearly, both $\{R_k\}$ and $\{C_k\}$ are i.i.d.\ uniform on $[n]$. Moreover,
$\{R_k\}$, $\{C_k\}$ and $S$ are mutually independent.

The processes $N_{ij}$ can be recovered from $S$ and the labels via
\[
N_{ij}(t)=\sum_{k:\sig_k\le t}1_{\{(R_k,C_k)=(i,j)\}}.
\]
Consequently, $U$ and $V$ can be written as
\[
U_i(t)=\sum_{k:\sig_k\le t}1_{\{R_k=i\}},
\qquad
V_j(t)=\sum_{k:\sig_k\le t}1_{\{C_k=j\}}.
\]

To construct $Z=Z^{(n)}$, let $T=T^{(n)}$ be another rate-$\beta n^2$ Poisson process, independent of $\{N_{ij}\}$. Let $\tau_k$, $k\ge1$ be its jump times in increasing order. Let
\[
Z_j(t)=\sum_{k:\tau_k\le t}1_{\{C_k=j\}}.
\]
Then obviously, $Z$ has same law as $V$. Moreover, because $Z$ is constructed in terms of $T$ and $\{C_k\}$ alone, while $U$, as shown above, can be expressed in terms of $S$ and $\{R_k\}$ alone, the independence of $(T,\{C_k\})$ and $(S,\{R_k\})$ implies that of $U$ and $Z$.
Note that the exchangeability of the $n$ pairs $(V_j,Z_j)$, $j\in[n]$, is clear from the construction.

It remains to show the estimate on $Z_1-V_1$.
Let $\pi(m)=\sum_{k=1}^m1_{\{C_k=1\}}$. Then
\[
V_1(t)=\pi(S(t)),
\qquad
Z_1(t)=\pi(T(t)).
\]
Thus
\[
\|Z_1-V_1\|_t=\sup_{s\le t}|\pi(S(s))-\pi(T(s))|.
\]
Let $\Theta=\Theta(n,t)=\|T-S\|_t$. Fix $0<\al<\kap\w 2$. For a bound on $\Theta$, note that $T-S$ is a martingale with jumps of size one and quadratic variation $T+S$. By a standard exponential martingale inequality for compensated Poisson processes, there are constants
$c_1,c_2>0$, depending on $t$ and $\beta$, such that
\[
\PP(\Theta>n^{1+\al/2})\le c_1e^{-c_2n^\al}.
\]
A Chernoff bound for $S(t)$, equivalently $T(t)$, gives $\PP(S(t)>2\beta tn^2)\le c_1e^{-c_2n^2}$. Hence, the event
\[
A_n=\Big\{\max(S(t),T(t))\le 2\beta t n^2,\ \Theta\le n^{1+\al/2}\Big\}
\]
satisfies
\[
\PP(A_n^c)\le c_1e^{-c_2n^\al}.
\]
Let $l_n=\lceil n^{1+\al/2}\rceil$. On $A_n$,
\[
\|Z_1-V_1\|_t
\le H(n,t)
:=\sup\{\pi(r+l_n)-\pi(r):0\le r\le 2\beta tn^2\}.
\]
For fixed $r$, $\pi(r+l_n)-\pi(r)$ is binomial with parameters $(l_n,n^{-1})$.
Hence, by a Chernoff bound,
\[
\PP(\pi(r+l_n)-\pi(r)>n^\al)
\le e^{-n^\al+e n^{\al/2}}
\le e^{-\frac12 n^\al}
\]
for all sufficiently large $n$. Since there are at most $c n^2$ possible values
of $r$, it follows that
\[
\PP(H(n,t)>n^\al)
\le c n^2 e^{-\frac12 n^\al}.
\]
Moreover, $H(n,t)\le l_n$. Therefore
\[
\E H(n,t)
\le n^\al+l_n\,\PP(H(n,t)>n^\al)
\le n^\al+o(1).
\]
Finally,
\[
\|Z_1-V_1\|_t
\le H(n,t)+(V_1(t)+Z_1(t))1_{A_n^c}.
\]
Since $V_1(t)$ and $Z_1(t)$ are rate-$\beta n$ Poisson random variables,
\[
\E[(V_1(t)+Z_1(t))^2]\le c n^2.
\]
By Cauchy--Schwarz and the exponential bound on $\PP(A_n^c)$,
\[
\E[(V_1(t)+Z_1(t))1_{A_n^c}]
\le c n\,\PP(A_n^c)^{1/2}=o(1).
\]
Consequently,
\[
\E\|Z_1-V_1\|_t\le c n^\al+o(1).
\]
Since $\al<\kap$, this gives
\[
\E[n^{-\kap}\|Z_1-V_1\|_t]\to0,
\]
as required.
\qed

Based on this we can present the following.

\noi{\bf Proof of Lemma \ref{lem21}.}
Consider the processes $\check\Pi_{ij}$, $i,j\in[n]$, an $n\times n$ array of i.i.d.\ rate-$p$ Poisson processes. Apply Lemma \ref{lem14} to this array with $\beta=p$, denoting the corresponding tuple $(Z_i)_{i\in[n]}$ by $(\Pi^*_i)_{i\in[n]}$. Let
\[
\gPi_i(t)=n^{-1/2}(\Pi^*_i(t)-pnt), \qquad i\in[n].
\]
Then the components are i.i.d., and this tuple is independent of $(\hat\Pi_{0i},\hat\Pi_{i0},\hat\Pi^\rs_i)_{i\in[n]}$. Moreover, recalling the definition of $\check\Pi_{ij}$ and denoting by $\check\Pi^\cs_i$ the column-sum $\sum_{j\in[n]}\check\Pi_{ji}$, we have $\hat\Pi^\cs_i(t)=\check\Pi^\cs_i(t)-pnt$. Thus
\[
\gPi_1-\hat\Pi^\cs_1=n^{-1/2}(\Pi^*_1-\check\Pi^\cs_1).
\]
Taking now $\kap=1/2$ in Lemma \ref{lem14} shows $\E[\|\gPi_1-\hat\Pi^\cs_1\|_t]\to0$.
\qed

\subsection{Proof of Theorem \ref{th1}}
\label{sec33}

%\noi{\bf Proof of Theorem \ref{th1}.}
Going back to equations \eqref{x39}-\eqref{x42}, our first goal is to couple $\{\hat Y^n_i\}_{i\in[n]}$, with i.i.d.\ $(\hat\gamma,2)$-BMs $(\tilde Y_i)_{i\in\N}$, in such a way that, for each $n$, $\hat Y_i-\tilde Y_i$ are exchangeable and $\hat Y^n_1\to\tilde Y_1$ in probability as $n\to\iy$. To this end, rewrite $(\hat Y_i)_{i\in[n]}$, as
\begin{align}\label{x51}
\hat Y_i&=\gW_i+\hat\gam_n\iota+\hat E_i,
\\ \notag
\gW_i&:=\hat\Pi_{0i}-\hat\Pi_{i0}-\hat\Pi^\rs_i+\gPi_i,
\\ \notag
\hat E_i&:=\hat\Pi^\cs_i-\gPi_i+\hat M_i.
\end{align}
By Lemma \ref{lem21}, the replacement of $\hat\Pi_i^\cs$ by $\gPi_i$ makes the processes $\gW^i$ i.i.d.\ martingales, while the error is absorbed into $\hat E_i$.
The quadratic variation of $\gW_1$ is
\[
[\gW_1]=n^{-1}(\check\Pi_{01}+\check\Pi_{10}+\check\Pi^\rs_1+\Pi^*_1)
\]
where $\check\Pi^\rs_i=\sum_j\check\Pi_{ij}$. The expression in parentheses is a Poisson process of rate $\gam_nn+qn+2pn$. Since $\gamma_n\to q$, we have that $[\gW_1]\to2\iota$ in probability, and thus $\gW_1$ converges in distribution to a $(0,2)$-BM.
Moreover, by Lemmas \ref{lem20} and \ref{lem21}, the processes $\hat E_i$ are exchangeable and, in view of Lemmas \ref{lem20} and \ref{lem21},
\begin{equation}\label{v1}
\E[\|\hat E_1\|_t]\to0.
\end{equation}

We now invoke Skorohod's representation to realize the prelimit and limit processes on a common space. With a slight abuse of notation, we continue to denote the prelimit processes by the same symbols as before.
The construction uses the one-dimensional convergence
$\gW_1^n \Rightarrow \tilde W_1$ together with the estimate $\hat E_1^n \to 0$ to obtain a coupling of the full triangular array $(\gW_i^n,\tilde W_i,\hat E_i^n)_{i\in[n]}$, in such
a way that the joint law of $(\gW_i^n,\hat E_i^n)_{i\in[n]}$ is preserved.

In particular, since $D^{(1)}$ is a Polish space, $\gW_1$ can be coupled with a $(0,2)$-BM $\tilde W_1$ in such a way that $\gW_1\to\tilde W_1$ in probability. Let $K_n(w,dy)$ be a transition kernel on $(D^{(1)},\calD^{(1)})$ such that $K_n(w,\cdot)$ is a version of the regular conditional law $\calL(\gW_1|\tilde W_1=w)$ for $\calL(\tilde W_1)$-a.e.\ $w$ \cite[Section 10.4]{bogachev2007}.
Let $G_n(\bar w,d\bar e)$ be a transition kernel on $(D^{(n)},\calD^{(n)})$ such that $G_n(\bar w,\cdot)$ is a version of the regular conditional law $\calL(\{\hat E_i\}_{i\in[n]}|\{\gW_i\}_{i\in[n]}=\bar w)$ for $\calL(\{\gW_i\})$-a.e.\ $\bar w$.

Now construct on one probability space the following objects. Let $(\tilde W_i)_{i\in\N}$ be i.i.d.\ $(0,2)$-BMs. Construct $\{\gW^n_i\}_{i\in[n]}$ so that, conditionally on $\{\tilde W_j\}_{j\in\N}$, they are independent with
\[
\calL(\gW^n_i|\{\tilde W_j\})=\calL(\gW^n_i|\tilde W_i)=K_n(\tilde W_i,\cdot), \qquad i\in[n].
\]
This makes $\{\gW^n_i\}_{i\in[n]}$ independent and preserves their original law.
Since $\gW_1^n\to\tilde W_1$ in probability under this coupling, and since
$\sup_n \E\|\gW_1^n\|_t^2<\infty$ and $\E\|\tilde W_1\|_t^2<\infty$,
the family $\{\|\gW_1^n-\tilde W_1\|_t\}_n$ is uniformly integrable. Hence
\begin{equation}\label{v2}
\E\|\gW_1^n-\tilde W_1\|_t\to0.
\end{equation}

Next, construct $\{\hat E^n_i\}_{i\in[n]}$ so that their conditional law given $\{\gW^n_i\}_{i\in[n]}$ is $G_n(\{\gW^n_i\}_{i\in[n]},\cdot)$. This preserves the joint law of $\{\gW^n_i,\hat E^n_i\}_{i\in[n]}$.
Finally, construct $\{\hat X^n_{i0}\}_{i\in[n]}$ according to their original law, independently of all $\tilde W_i$, $\gW^n_i$ and $\hat E^n_i$.

In agreement with \eqref{x51}, on the new probability space, set $\hat Y^n_i=\gW^n_i+\hat\gam_n\iota+\hat E^n_i$. Recalling relation \eqref{x42} and the uniqueness statement of Lemma \ref{lem22}, apply that lemma with $\beta=-p$ and $w_i=\hat X^n_{0i}+\hat Y^n_i$ to define $(\hat X^n,\hat L^n)$ as the unique solution to the system of equations
\[
(\hat X^n_i,\hat L^n_i)=\Gamma(\hat X^n_{0i}+\hat Y^n_i-p\lan\hat L^n\ran),
\qquad i\in[n].
\]
This preserves the original law of $(\hat X^n,\hat L^n)$.

The next step is to construct the Brownian system and estimate the distance between the two systems. Let
\[
\tilde Y_i=\tilde W_i+\hat\gamma\iota,
\qquad i\in[n],
\]
and let $(\tilde X^n,\tilde L^n)$ be the unique (again by Lemma \ref{lem22}) tuple satisfying
\[
(\tilde X^n_i,\tilde L^n_i)=\Gamma(\hat X^n_{0i}+\tilde Y_i-p\lan\tilde L^n\ran),\qquad i\in[n].
\]
We now use the $1$-Lipschitz property of $\Gamma_2$ stated in \eqref{x52}. Namely, for any $t>0$,
\begin{align}\label{x53}
\notag
\|\hat L^n_i-\tilde L^n_i\|_t &\le\| (\hat X^n_{0i}+\hat Y^n_i-p\lan \hat L^n\ran)-(\hat X^n_{0i}+\tilde Y_i-p\lan \tilde L^n\ran)\|_t
\\
&\le\|\gW^n_i-\tilde W_i\|_t+|\hat\gam_n-\hat\gam|t+\|\hat E^n_i\|_t+p\lan\|\hat L^n-\tilde L^n\|_t\ran.
\end{align}
Averaging,
\[
\lan\|\hat L^n-\tilde L^n\|_t\ran\le q^{-1}(\lan\|\gW^n-\tilde W\|_t\ran+|\hat\gam_n-\hat\gam|t+\lan\|\hat E^n\|_t\ran).
\]
Substituting back in \eqref{x53},
\[
\|\hat L^n_i-\tilde L^n_i\|_t\le \|\gW^n_i-\tilde W_i\|_t+(1+pq^{-1})|\hat\gam_n-\hat\gam|t+\|\hat E^n_i\|_t+pq^{-1}(\lan\|\gW^n-\tilde W\|_t\ran+\lan\|\hat E^n\|_t\ran).
\]
By \eqref{v1} and \eqref{v2}, the expression on the right converges to zero in $L^1$.
Hence $\E[\|\hat L^n_1-\tilde L^n_1\|_t]\to0$. Using
\[
\hat X^n_i-\tilde X^n_i=(\hat X^n_{0i}+\hat Y^n_i-p\lan \hat L^n\ran+\hat L^n_i)-(\hat X^n_{0i}+\tilde Y_i-p\lan \tilde L^n\ran+\tilde L^n_i),
\]
and once again \eqref{x53}, shows that $\E[\|\hat X^n_1-\tilde X^n_1\|_t]\to0$.

The convergence of a $k$-tuple to $k$ independent copies of $(\oo X,\oo L)$ now follows immediately from that of the Brownian system, stated in Theorem \ref{th2}.

The convergence of the empirical measures also follows now from Theorem \ref{th2}, but some further details are required. Let $\hat Z^n_i=(\hat X^n_i,\hat L^n_i)$ and $\tilde Z^n_i=(\tilde X^n_i,\tilde L^n_i)$. Recall that we have $\tilde\mu^n_t\to\mu_t$ in $C(\R_+,\calP(\R^2))$ in probability.
The pairs $(\hat Z_i^n,\tilde Z_i^n)_{i\in[n]}$ are exchangeable and one has $\|\hat Z^n_1-\tilde Z^n_1\|_t\to0$ in probability.
Fix $t>0$ and $\eps>0$. By exchangeability,
\[
\frac1n\E\sum_{i\in[n]} 1_{\{\|\hat Z_i^n-\tilde Z_i^n\|_t>\eps\}}
=
\PP(\|\hat Z_1^n-\tilde Z_1^n\|_t>\eps)\to 0.
\]
Hence
$\frac1n\sum_{i\in[n]} 1_{\{\|\hat Z_i^n-\tilde Z_i^n\|_t>\eps\}}\to0$ in probability.
Now take a $1$-Lipschitz, $1$-bounded function $f:\R^2\to\R$. Then for $s\le t$,
\[
\left|\int fd \mu^n_s-\int fd\tilde \mu^n_s\right|
\le
\frac1n\sum_{i\in[n]} |f(\hat Z_i^n(s))-f(\tilde Z_i^n(s))|
\le
\eps+\frac2n\sum_{i\in[n]} 1_{\{\|\hat Z_i^n-\tilde Z_i^n\|_t>\eps\}}.
\]
Noting that the right side does not depend on $s$,
and taking $n\to\iy$ followed by $\eps\to0$, shows that the bounded-Lipschitz distance satisfies
\[
\sup_{s\le t}d_{\mathrm{BL}}(\mu^n_s,\tilde\mu^n_s)\to0 \text{ in probability}.
\]
The result now follows by the triangle inequality.

As for the convergence of $\la^n$, we have $\|\la^n-\tilde\la^n\|_t=\|\lan\hat L^n-\tilde L^n\ran\|_t\le\lan\|\hat L^n-\tilde L^n\|_t\ran\to0$ in probability, hence again the result follows from Theorem \ref{th2}.
\qed

\subsection*{Acknowledgment} Research was supported by ISF grant 3240/25.

%\footnotesize

%%\bibliographystyle{plain}
%%\bibliographystyle{annotate}
%%\bibliographystyle{apalike}
\bibliographystyle{is-abbrv}

\bibliography{main}

\vspace{.5em}

\end{document}